\documentclass{amsart}

\usepackage{mathrsfs,amsmath,amssymb,amsthm,bbm,cite}
\usepackage{hyperref}

\newtheorem{lem}{Lemma}[section]
\newtheorem{thrm}[lem]{Theorem}
\newtheorem{prop}[lem]{Proposition}

\theoremstyle{definition}

\theoremstyle{remark}

\newcommand{\eq}[2]{\begin{equation}\label{#1}#2\end{equation}}

\renewcommand{\Re}{\ensuremath{\operatorname{Re}}}
\renewcommand{\Im}{\ensuremath{\operatorname{Im}}}
\renewcommand{\epsilon}{\varepsilon}

\newcommand{\mb}{\mathbf}
\newcommand{\mr}{\mathrm}
\newcommand{\mf}{\mathfrak}

\numberwithin{equation}{section}
\allowdisplaybreaks

\newcommand{\R}{\ensuremath{\mathbb{R}}}
\newcommand{\C}{\ensuremath{\mathbb{C}}}
\newcommand{\D}{\ensuremath{\mathbb{D}}}
\newcommand{\Z}{\ensuremath{\mathbb{Z}}}

\newcommand{\Schwartz}{\ensuremath{\mathscr{S}}}

\DeclareMathOperator{\sgn}{sgn}
\newcommand{\<}{\ensuremath{\langle}}
\renewcommand{\>}{\ensuremath{\rangle}}

\newcommand{\p}{\ensuremath{\partial}}

\newcommand{\bpf}{\begin{proof}}
\newcommand{\epf}{\end{proof}}
\newcommand{\qtq}[1]{\quad\text{#1}\quad}

\DeclareMathOperator{\tr}{tr}
\let\det=\undefined\DeclareMathOperator{\det}{det}

\newcommand{\op}{\mr{op}}

\newcommand{\I}{\mf I}

\newcommand{\norm}[1]{{\vert\kern-0.25ex\vert\kern-0.25ex\vert#1\vert\kern-0.25ex\vert\kern-0.25ex\vert}}
\newcommand{\bbo}{\mathbbm 1}

\DeclareMathOperator{\arcsinh}{arcsinh}

\title[Microscopic conservation laws for integrable lattice models]{Microscopic conservation laws\\for integrable lattice models}

\author[B.~Harrop-Griffiths]{Benjamin Harrop-Griffiths}
\address{Benjamin Harrop-Griffiths\\
Department of Mathematics\\
University of California, Los Angeles, CA 90095, USA}
\email{harropgriffiths@math.ucla.edu}

\author[R.~Killip]{Rowan Killip}
\address{Rowan Killip\\
Department of Mathematics\\
University of California, Los Angeles, CA 90095, USA}
\email{killip@math.ucla.edu}

\author[M.~Vi\c san]{Monica Vi\c{s}an}
\address{Monica Vi\c{s}an\\
Department of Mathematics\\
University of California, Los Angeles, CA 90095, USA}
\email{visan@math.ucla.edu}

\begin{document}

\begin{abstract}
We consider two discrete completely integrable evolutions: the Toda Lattice and the Ablowitz--Ladik system.  The principal thrust of the paper is the development of microscopic conservation laws that witness the conservation of the perturbation determinant under these dynamics.  In this way, we obtain discrete analogues of objects that we found essential in our recent analyses of KdV, NLS, and mKdV.

In concert with this, we revisit the classical topic of microscopic conservation laws attendant to the (renormalized) trace of the Green's function.
\end{abstract}

\maketitle

\section{Introduction}

A typical property of completely integrable systems is that they can be expressed via a Lax pair or zero curvature condition:
\begin{align}\label{E:LP}
\tfrac{d\ }{dt} L(t;\kappa) = [ P(t;\kappa), L(t,\kappa) ].
\end{align}
Here $L$ and $P$ are typically operator pencils in the spectral parameter $\kappa$.  Such will be the case for the concrete models we discuss, namely, the Korteweg--de Vries equation \eqref{KdV}, the nonlinear Schr\"odinger equation \eqref{NLS}, the (complex) modified Korteweg--de Vries equation \eqref{mKdV}, the Toda Lattice \eqref{TL}, and the Ablowitz--Ladik system \eqref{AL eqn}.  At this moment, however, we do not wish to narrow our focus unnecessarily.  The question that stimulated this paper remains interesting for any system admitting the representation \eqref{E:LP}. 

The evolution \eqref{E:LP} guarantees that for each $\kappa$, the conjugacy class of $L(t;\kappa)$ is independent of time.  In particular, class functions such as the trace and determinant will be conserved --- if they are defined!   For many models, however, they are not;  it is necessary to renormalize.

The basic form of renormalization rests on a trivial static solution of the integrable system.  For typical models, this would be the zero solution; we deliberately choose not to specify it as such, because in some models (such as spin chains) there is no such thing as the zero solution.

With these ideas set, it is easy to imagine that we have the following:
\begin{gather}\label{conserve det}
\tfrac{d\ }{dt} \log \det\bigl[ L(t;\kappa) / L_0(\kappa) \bigr] = 0,  \\[1mm]
\tfrac{d\ }{dt} \tr \bigl\{ L(t;\kappa)^{-1}  -  L_0(\kappa)^{-1} \bigr\} = 0. \label{conserve tr}
\end{gather}
Here $L_0(\kappa)$ stands for the (time-independent) Lax operator associated to the chosen trivial solution.  We take the logarithm in \eqref{conserve det} because this has proved to be the more useful quantity (compared to the pure determinant).

We merely `imagine' that \eqref{conserve det} and \eqref{conserve tr} hold because, while this is easily made rigorous in the case of finite-dimensional Hamiltonian ODE with corresponding finite-dimensional Lax operators, there are non-trivial analytical obstacles to be overcome in the infinite-dimensional setting:  Are the dynamics well-posed? Are the trace or determinant well-defined?
 
The claims \eqref{conserve det} and \eqref{conserve tr} may be termed \emph{macroscopic} conservation laws, mirroring notions such as the conservation of energy.  In this paper, we seek \emph{microscopic} conservation laws, which provide a local-in-space explanation of conservation: energy is conserved globally because the time derivative of the energy density is the divergence of (minus) the energy current.

While macroscopic conservation laws can be recovered from microscopic versions by integrating the density over the whole space, the key advantage of the latter stems from the possibility of interposing a weight.  We know of no better witness to the power of this idea than Kato's seminal local smoothing estimate for the Korteweg--de Vries equation \cite{Kato}:  Solutions $q(t,x)$ to
\begin{equation}\label{Kato:mcl}
q_t = - q_{xxx} + 6 qq_x \qtq{also satisfy} \bigl(q^2\bigr)_t + \bigl(-3(q_x)^2-4q^3+(q^2)_{xx}\bigr)_x = 0.
\end{equation}
Integrating the latter against a sigmoid function, Kato proved
\begin{equation}\label{Kato:smooth}
\int_0^1\int_0^1 |q_x(t,x)|^2\,dx\,dt \lesssim \| q(0) \|_{L^2}^2 + \| q(0) \|_{L^2}^6.
\end{equation}
This gain of one spatial derivative explains the `smoothing effect' appellation.  The restriction to a finite spatial window is essential here; the Hamiltonian nature of the equation shows there can be no global gain in regularity.

A crucial ingredient in Kato's argument is the coercivity of the dominant term, $-3(q_x)^2$, in the current.  Coercivity has been an important factor in the choice and deployment of such local conservation laws in the recent papers \cite{BKV,KV,HGKV} on continuum models.  Moreover, the laws we employed in those papers remain meaningful at low-regularity, including for $H^{s}$-valued solutions with $s$ non-integer and/or negative.  By comparison, the traditional polynomial conservation laws are useful only for integer values of $s\geq 0$.

Another benefit of microscopic conservation laws over many of their macroscopic counterparts is that they remain meaningful for solutions that do not decay at infinity.  The recent paper \cite{KMV}, which establishes the invariance of white noise under \eqref{KdV}, gives one example of how such laws can be deployed in this way. 

One microscopic representation of the conservation law \eqref{conserve tr} presents itself immediately: the trace is fundamentally an integral (or sum) over space and the conserved density may be taken to be the diagonal entries of the Green's function.  This is by no means a new observation.  This idea and many fruitful consequences may be found throughout the literature on integrable systems.  One early application, \cite{MR0508337}, is to the derivation of polynomial conservation laws for KdV from the asymptotic expansion of the diagonal Green's function (as $\kappa\to\infty$).

Likewise, a systematic approach to the polynomial conservation laws for the Toda and Ablowitz--Ladik evolutions that illustrates their connection to the diagonal Green's function, has been developed in \cite{GH-Norway,MR2416645}.  

We should pause to note that the diagonal Green's function is not the only microscopic representation of \eqref{conserve tr}.  Indeed, any microscopic conservation law may freely be modified by any divergence-free vector field in spacetime.  But then how is one to choose the `correct' representative?  We do not claim that there is a unique answer; nevertheless, our investigations in this direction have led us to value coercivity and quantities amenable to operator-theoretic analysis.  

At last, we reach the central question of this paper: Can we find microscopic representations of \eqref{conserve det}?  For models such as KdV and NLS, it has been shown that the perturbation determinant coincides with the reciprocal of the transmission coefficient.  In this guise, the conservation law has been studied extensively.  In \cite{MR0303132}, for example, a microscopic representation based on the Jost solution is demonstrated (in the KdV setting) and shown to be a generating function for the polynomial conservation laws.  The papers \cite{GH-Norway,MR2416645} carry the same philosophy over to the discrete case.   In this way, one may say that our central problem has been solved.  However, as we have discussed earlier, microscopic representations of a macroscopic conservation law are far from unique.  In this paper, we will be presenting a \emph{different} solution, one that is informed by successes in applying these laws to the well-posedness problem.

The very idea of Jost solutions is already restricted to the class of integrable models whose Lax operators admit a scattering theory.  Traditionally, at least, this presupposes infinite volume with rapidly decreasing initial data.  The Green's function and the perturbation determinant transcend such restrictions.  In \cite{KV,MR3820439}, for example, we see the seamless manner in which they can be applied to \eqref{KdV} posed both on the circle and on the line.  By comparison, we see in \cite{MR3874652}, for example, how involved the analysis of low regularity conservation laws becomes when based on the Jost solutions.

The reader may well ask why we seek microscopic representations of \eqref{conserve det}.  Why are we not satisfied with \eqref{conserve tr} and the diagonal Green's function? Our answer stems from a number of empirical observations we made while working on the well-posedness problem for \eqref{KdV}, \eqref{NLS}, and \eqref{mKdV}.  We observed that our microscopic conservation laws attendant to \eqref{conserve det}, which will be reviewed in Sections~\ref{S:KdV} and~\ref{S:NLS}, showed better coercivity than enjoyed by the diagonal Green's function.  This was crucial in our analyses in \cite{BKV,KV,HGKV}.  Moreover, we found that we are able to express these new microscopic conservation laws via nonlinear functions of the diagonal Green's function; thus there is no new overhead of complexity.  Furthermore, the functional derivatives of macroscopic conservation laws also enter into our analysis.  In the case of the perturbation determinant this leads to the diagonal Green's function --- which is already needed in the investigation.  By comparison, the functional derivative of the trace of the resolvent involves the diagonal entries of the square of the resolvent; this involves the whole Green's function and so adds another layer of complexity.

In this paper, we will present microscopic conservation laws for the Toda Lattice (Section~\ref{S:Toda}) and Ablowitz--Ladik system (Section~\ref{S:AL}) that mimic those that we found so useful in our investigations of continuum models.  Moreover, we will see that they share the favorable coercivity properties of their continuum analogues.

In fact, the conservation laws derived in this paper may be seen as more fundamental than their continuum counterparts.  Early in the history of these lattice models, it was discovered that they could reproduce continuum models in certain limiting regimes; see, e.g., \cite{AL1,AL2,Toda1970}.  It is not difficult to verify that under the formal limiting processes outlined in these papers, the discrete conservation laws obtained in this paper reduce to their continuum counterparts.  On the other hand, just as there is no systematic method of deducing completely integrable discrete models from their continuum analogues, so there is no way of passing from the continuum laws to their discrete counterparts.  Rather, we were forced to reason by analogy and employ a little trial and error.

This question of passing to the continuum limit touches upon one of our motivations in developing the results presented in this paper.  One driving force behind developing a low-regularity theory of partial differential equations is to handle the small but wild oscillations attendant to any system in (or near) thermal equilibrium.  Concretely, one may ask how the temperature of a crystal affects the propagation of large-scale solitary waves.  To tackle this type of question, which rests on taking a continuum limit for solutions with no smoothness whatsoever, it seems necessary to develop discrete manifestations of tools with proven efficacy in the low-regularity continuum regime. 

When embarking on this project, we also hoped to find a simple unifying principle that would essentially automate the construction of microscopic laws attendant to \eqref{conserve det} for general Lax operators.  This ambition has not been realized.  Indeed, the results presented here only serve to indicate that the matter is rather more subtle than we had initially hoped.  The diagonal Green's function is ubiquitous; however, the plethora of nonlinear ways in which it manifests seems to defy a universal explanation.

The paper is organized as follows: In Sections~\ref{S:KdV} and~\ref{S:NLS}, we recapitulate results for the KdV and NLS hierarchies, respectively.  This sets the stage for our treatments of the Toda Lattice in Section~\ref{S:Toda} and the Ablowitz--Ladik system in Section~\ref{S:AL}.

\subsection*{Acknowledgements} R.~K. was supported by NSF grant DMS-1856755 and M.~V. by grant DMS-1763074.

\section{KdV}\label{S:KdV} 

In this section, we discuss known microscopic conservation laws for the Korteweg--de Vries hierarchy.  While these results have proven crucial in proving low-regularity well-posedness results \cite{BKV,KMV,KV}, we are content to present the key identities in the context of Schwartz-class solutions.  Indeed, the low-regularity cases are most easily treated by continuous extension from that case.

The Korteweg--de Vries equation
\begin{align}\label{KdV}\tag{KdV}
q_t = - q_{xxx} + 6 qq_x
\end{align}
is a Hamiltonian equation with respect to the Poisson structure
\begin{align}\label{E:PS}
\{ F, G \} := \int  \frac{\delta F}{\delta q}(x) \biggl(\frac{\delta G}{\delta q}\biggr) '(x) \,dx.
\end{align}

As Lax operator, we adopt the standard choice
$$
L(\kappa) := -\partial_x^2 + q+\kappa^2,
$$
interpreted as an unbounded self-adjoint operator on $L^2(\R)$ with form domain $H^1(\R)$.  The second member of the Lax pair is
$$
P := - 4 \partial_x^3 + 3\bigl(\partial_x q + q\partial_x \bigr).
$$

For $\kappa$ sufficiently large (depending on $q$), the Lax operator $L(\kappa)$ is invertible and the inverse is represented by a continuous kernel, the Green's function.  The object of central interest in this section is the diagonal Green's function:
$$
g(x;\kappa,q) := \bigl\langle\delta_x , L(\kappa)^{-1} \delta_x\bigr\rangle .
$$
When $q\equiv 0$, which is the natural static solution in this hierarchy, we have the Lax operator $L_0=-\partial_x^2 + \kappa^2$. The corresponding diagonal Green's function is $g_0\equiv\frac1{2\kappa}$.

With these preliminaries set, we have the following identities:
\begin{align}
\tr \bigl\{ L(t;\kappa)^{-1}  -  L_0(\kappa)^{-1} \bigr\} &= \int g(x;\kappa,q) - \tfrac{1}{2\kappa}\,dx \label{KdV tr} \\
-\log \det\bigl[ L(t;\kappa) / L_0(\kappa) \bigr] &= \int  \kappa -\tfrac{1}{2g(x;\kappa,q)}\,dx \label{KdV det}
\end{align}
valid for $q\in\Schwartz$ and $\kappa\gg 1$.  Both of these lead readily to microscopic conservation laws.  However, both also prove unsatisfactory: convergence requires too much decay.  For example, if $q\leq 0$, then $q\in L^1$ is required.

Our remedy is to use that $\int q\,dx$ is preserved under the KdV hierarchy to renormalize in a prudent way:

\begin{prop}\label{P:KdV}
Given $q\in\Schwartz$ and $\kappa\gg 1$, we define
\begin{align*}
\gamma(x;\kappa) := g(x;\kappa) - \tfrac{1}{2\kappa} + \tfrac{1}{4\kappa^2} e^{-2\kappa|\cdot|} * q
	\qtq{and}
\rho(x;\kappa) := \kappa -\tfrac{1}{2g(x;\kappa)} + \tfrac{1}{2} e^{-2\kappa|\cdot|} * q .
\end{align*}
Both constitute conserved densities under the KdV hierarchy.  For example, under the KdV flow, $q_t = -q_{xxx}+6qq_x$, we have
\begin{gather}
\gamma_t = \Bigl(  -\gamma_{xx} + \tfrac{3}{4\kappa^2} \bigl[e^{-2\kappa|\cdot|} * q^2\bigr]  + 3g_{xx} - 6qg-12\kappa^2g+6\kappa \Bigr)_x , \label{KdV gamma dot}\\
\rho_t = \Bigl(  \tfrac{3}{2} \bigl[e^{-2\kappa|\cdot|} * q^2\bigr] +  2 q \bigl[\kappa -\tfrac{1}{2g}\bigr] -4\kappa^2\rho  \Bigr)_x .\label{KdV rho dot}
\end{gather}
\end{prop}

All of these claims are developed from first principles in \cite{KV}.  The renormalization of the diagonal Green's function chosen here matches the head of a Neumann series expansion:
$$
\gamma(x;\kappa,q) = \bigl\langle\delta_x , \bigl[ L(\kappa)^{-1}  -  L_0(\kappa)^{-1} + L_0(\kappa)^{-1} q  L_0(\kappa)^{-1} \bigr] \delta_x\bigr\rangle.
$$
It guarantees that $\gamma\in L^1$ for $\kappa$ large, even when $q\in H^{-1}$.  In this way it is superior to simply subtracting a numerical multiple of $q$ itself.

We have not seen the connection between the perturbation determinant and the reciprocal of $g$ observed prior to \cite{KV}.  This renormalization was also introduced there; it not only guarantees that $\rho\in L^1$ for $q\in H^{-1}$, but even ensures that $\rho(x)$ is a non-negative convex function of $q$ for each choice of $x$.  

\section{NLS \& mKdV}\label{S:NLS} 

In this section we review the microscopic conservation laws used in our work \cite{HGKV} on optimal well-posedness for the nonlinear Schr\"odinger equation
\eq{NLS}{\tag{NLS}
i\frac d{dt}q = - q'' \pm  2|q|^2q
}
and the (Hirota)  complex modified Korteweg--de Vries equation
\eq{mKdV}{\tag{mKdV}
\frac d{dt}q = -  q''' \pm 6|q|^2q'.
}
Both \eqref{NLS} and \eqref{mKdV} are Hamiltonian equations with respect to the following Poisson structure (written using Wirtinger derivatives)
\begin{align*}
\{F,G\} := \tfrac1i \int  \tfrac{\delta F}{\delta q}\tfrac{\delta G}{\delta r} - \tfrac{\delta F}{\delta r}\tfrac{\delta G}{\delta q} \,dx,
\end{align*}
where \(F,G\colon \Schwartz\rightarrow \R\) and \(r := \pm \bar q\).

Lax pairs for \eqref{NLS} and \eqref{mKdV} were introduced in \cite{MR450815,MR0406174}.  Our conventions in \cite{HGKV} were
\begin{equation}\label{Intro AKNS L}
L(\kappa) := \begin{bmatrix}\kappa - \p & q\\-r&\kappa + \p\end{bmatrix} \qtq{and} P:=i \begin{bmatrix}2\p^2-qr & -q\p-\p q\\ r\p+\p r & -2\p^2+qr\end{bmatrix}
\end{equation}
for \eqref{NLS} and 
\begin{align*}
L(\kappa) := \begin{bmatrix}\kappa - \p & q\\-r&\kappa + \p\end{bmatrix} \qtq{and}
P:=\begin{bmatrix} - 4\p^3 + 3 qr\p + 3 \p qr  & 3q'\p +3\p q' \\ 3r'\p+3\p r' & -4\p^3 + 3 qr\p + 3 \p qr \end{bmatrix}
\end{align*}
for \eqref{mKdV}.

By direct computation, one finds that
\[
L_0(\kappa)^{-1} = \begin{bmatrix}(\kappa - \p)^{-1} & 0\\0&(\kappa + \p)^{-1}\end{bmatrix}
\]
admits the integral kernel 
\begin{equation}\label{G_0}
G_0(x,y;\kappa) = e^{-\kappa|x - y|}\begin{bmatrix}\bbo_{\{x<y\}}&0\\0&\bbo_{\{y<x\}}\end{bmatrix} \quad\text{for $\kappa >0$}.
\end{equation}
For $\kappa$ sufficiently large depending on $q$, $L(\kappa)$ is invertible as an operator on $L^2(\R;\C^2)$ and the inverse $L(\kappa)^{-1}$ admits a matrix-valued integral kernel $G(x,y;\kappa)$.
Due to jump discontinuities, one cannot expect to restrict $G(x,y;\kappa)$ to the $x=y$ diagonal in a meaningful way.  However, both $G-G_0$ and \(G_{11} + G_{22}\) (where subscripts indicate matrix entries) are continuous functions of $(x,y)\in \R^2$.  This allows us to unambiguously define the continuous function
\begin{align*}
\gamma(x;\kappa) :=\tr\bigl\{[G-G_0](x,x;\kappa)\bigr\}   = [G_{11}+ G_{22}](x,x;\kappa)  - 1.
\end{align*}
By definition, we then have
\[
\tr \bigl\{ L(t,\kappa)^{-1}  -  L_0(\kappa)^{-1} \bigr\} = \int \gamma(x;\kappa)\,dx,
\]
whenever \(q\in \Schwartz\) and $\kappa$ is taken sufficiently large (depending on $q$).

In~\cite{HGKV}, we found a corresponding density for the perturbation determinant. For \(q\in \Schwartz\) and \(\kappa\) sufficiently large (depending on $q$), we have
\[
\log \det\bigl[ L(t;\kappa) / L_0(\kappa) \bigr] = \int \rho(x;\kappa)\,dx,
\]
where the density $\rho$ is given by
$$
\rho(x;\kappa):=\frac{q(x)g_{21}(x;\kappa)-r(x)g_{12}(x;\kappa)}{2+\gamma(x;\kappa)}
$$
with $g_{12}(x;\kappa) =G_{12}(x,x;\kappa)$ and $g_{21}(x;\kappa) =G_{21}(x,x;\kappa)$.

We then have the following microscopic conservations laws for \eqref{NLS} and \eqref{mKdV}:

\begin{prop}[\!\!{\cite{HGKV}}]\label{P:NLS}
Given $q\in\Schwartz$ and $\kappa$ sufficiently large (depending on $q$), the expressions $\gamma$ and $\rho$ constitute conserved densities under the ZS--AKNS hierarchy.  In particular, under the \eqref{NLS} and \eqref{mKdV} flows
$$
\partial_t \gamma= \partial_x {}^\gamma\!j_\star \qtq{and} \partial_t \rho= \partial_x j_\star
$$
where 
\begin{align*}
{}^\gamma\!j_{\mr{NLS}}&= -i\bigl(2rg_{12} - 2qg_{21} - 4\kappa\gamma \bigr),\\
{}^\gamma\!j_{\mr{mKdV}}&=-\gamma'' +12\kappa(rg_{12} - qg_{21}) - 12\kappa^2\gamma + 6qr(1 + \gamma),\\
j_{\mr{NLS}}&= i\left(\tfrac{ q'\cdot g_{21} + r'\cdot g_{12}}{2 + \gamma} - qr + 2\kappa \rho\right),\\
j_{\mr{mKdV}}&= -\tfrac{ (q'' - 2q^2r)\cdot g_{21} - (r'' - 2r^2q)\cdot g_{12}}{2 + \gamma} + q'r - qr' + 2i\kappa j_{\mr{NLS}}.
\end{align*}
\end{prop}

The main benefits of the microscopic conservation law for \(\rho\) are its superior coercivity properties, which were crucial for our applications in \cite{HGKV}. Denoting the quadratic (in \(q\)) terms in the currents by \(j_\star^{[2]}\), we have
\begin{align*}
\int \Im j^{[2]}_{\mr{NLS}}(x;\kappa)\,dx &= \mp \int \tfrac{2\xi^2|\hat q(\xi)|^2}{4\kappa^2+\xi^2}\, d\xi, \\
 \int \Re j^{[2]}_{\mr{mKdV}}(x;\kappa)\,dx &= \pm \int \tfrac{6\kappa\xi^2|\hat q(\xi)|^2}{4\kappa^2+\xi^2}\, d\xi.
\end{align*}
This coercivity of the currents was crucial in proving local smoothing estimates (cf.~\ref{Kato:smooth}). Moreover, the density itself is also coercive,
\[
\int \Re\rho^{[2]}\,dx = \pm \int \tfrac{2\kappa|\hat q(\xi)|^2}{4\kappa^2+\xi^2}\, d\xi,
\]
which we used to show that precompact sets of initial data produce tight ensembles of trajectories.  (This would be a trivial consequence of well-posedness; for us, it was a crucial step in \emph{proving} well-posedness.)

\section{Toda}\label{S:Toda} 

The Toda Lattice \cite{TodaBook} is a completely integrable chain of anharmonic oscillators.  The Hamiltonian takes the form
\begin{equation}\label{Toda H}
H := \sum_{n\in \Z}\Bigl[\tfrac12 p_n^2 + V(q_{n+1} - q_n)\Bigr]\qtq{where}V(x) := e^{-x} + x - 1.
\end{equation}
Here $q_n$ represent particle positions (relative to global equilibrium) and $p_n$ their conjugate momenta.  Correspondingly, the Poisson bracket is given by
\begin{equation}\label{Toda PB}
\{F,G\} := \sum_{n\in \Z}\Bigl[ \tfrac{\partial F}{\partial q_n}\tfrac{\partial G}{\partial p_n} - \tfrac{\partial F}{\partial p_n}\tfrac{\partial G}{\partial q_n}\Bigr]
\end{equation}
and the resulting dynamics takes the form
\begin{equation}\label{TL}\tag{TL}
\frac d{dt} q_n = p_n
	\qtq{and}
\frac d{dt} p_n = V'(q_{n+1}-q_{n}) - V'(q_n-q_{n-1}).
\end{equation}

A Lax pair representation of these dynamics was discovered by Flaschka~\cite{Flaschka}, which we will soon describe. The first step is to change variables to
\begin{equation}\label{Flaschka}
a_n := \tfrac12 e^{\frac12(q_n - q_{n+1})} \qtq{and} b_n := - \tfrac12 p_n.
\end{equation}
In these variables, the Hamiltonian becomes
\begin{equation}\label{TF H}
H = \sum_{n\in \Z}\Bigl[2b_n^2 + V( - 2 \log 2 a_n ) \Bigr] \qtq{where again}  V(x) = e^{-x} + x - 1, 
\end{equation}
the Poisson bracket takes the form
\begin{equation}\label{TF PB}
\{F,G\} = \tfrac14\sum\limits_{n\in\Z}a_n\Bigl[\tfrac{\partial F}{\partial a_n}\Bigl(\tfrac{\partial G}{\partial b_{n+1}} - \tfrac{\partial G}{\partial b_n}\Bigr) - \Bigl(\tfrac{\partial F}{\partial b_{n+1}} - \tfrac {\partial F}{\partial b_n}\Bigr)\tfrac{\partial G}{\partial a_n}\Bigr],
\end{equation}
and the equations of motion become
\begin{align}\label{Toda PDE}
\frac d{dt}a_n = a_n\bigl(b_{n+1} - b_n\bigr)
	\qtq{and}
\frac d{dt} b_n = 2\bigl(a_n^2 - a_{n-1}^2\bigr).
\end{align}

We shall confine our attention in this paper to finite-energy solutions to the Toda system.  In view of the strict convexity of $V(x)$ and its quadratic vanishing at $x=0$, we see that finite energy can be expressed by the following equivalent conditions:
\begin{equation}\label{Toda Espace}
\sum_{n\in\Z} p_n^2 + (q_{n+1}-q_n)^2 <\infty \qtq{or} \sum_{n\in\Z} b_n^2 + (\log 2a_n)^2 < \infty.
\end{equation}
With a view to our future needs, we define `balls' in the energy space via
\begin{equation}\label{Toda Bdelta}
B^\kappa_\delta := \bigl\{ (a_m,b_m) : H < \delta^2 \kappa \bigr\}.
\end{equation}

It is a trivial matter to see that the Toda Lattice is globally well-posed in the energy space, which makes it an ideal setup for our discussion.  Nevertheless, some calculations below will require stronger hypotheses, namely,
\begin{align}\label{Toda l1}
b_n\in\ell^1 \qtq{and} \log(2a_n)\in \ell^1.
\end{align}
In particular, these assumptions allow consideration of the conserved quantities
\begin{align}\label{Toda Casimirs}
M := \sum_{n\in \Z} (q_{n+1} - q_n) = - \sum_{n\in \Z} 2\log(2a_n) 
	\qtq{and}
P := \sum_{n\in \Z}p_n = - \sum_{n\in \Z} 2b_n.
\end{align}
These represent the net expansion of the lattice and total momentum, respectively.  Both are Casimirs:
$
\{ a_n, M \} = \{ b_n, M\} = 0 = \{ a_n, P \} = \{ b_n, P\}
$
for every $n\in\Z$.

The Lax representation discovered by Flaschka~\cite{Flaschka} takes the form
\begin{align}\label{Toda LaxE}
\tfrac d{dt}L_\pm = [\pm \mathcal P,L_\pm]
\end{align}
in terms of the tri-diagonal operators
\begin{align}
(L_\pm f)_n &:= \cosh(\kappa) f_n - \Bigl(a_nf_{n+1} + a_{n-1}f_{n-1} \pm b_nf_n\Bigr), \label{Toda L}\\
(\mathcal P f)_n &:= a_nf_{n+1} - a_{n-1}f_{n-1}. \label{Toda P}
\end{align}
As we consider finite energy solutions, these are bounded operators on $\ell^2(\Z)$.  Here, $\cosh(\kappa)$ serves as the spectral parameter; the merit of this representation will become apparent later when we discuss the Green's function (cf. \eqref{Toda G0} below).

The relationship between $L_+$ and $L_-$ is ultimately that of changing the sign of the spectral parameter.   Concretely, defining the unitary involution
$$
(U f)_n=(-1)^n f_n, \qtq{we find}  L_+ - 2 \cosh(\kappa)  = - U  L_- U .
$$
This overcomes the restriction that the $\cosh(\kappa)$ parameterization leads to the spectral parameter always being positive, while also ensuring that the conserved densities described below are non-negative.

The natural static solution of the Toda Lattice is the minimum energy state: $b_n\equiv 0$ and $a_n\equiv\tfrac12$.  In this state, $L_+$ and $L_-$ agree; we adopt the notation $L_0$:
\begin{align*}
(L_0f)_n =  \cosh(\kappa) f_n - \tfrac12\bigl(f_{n+1} + f_{n-1}\bigr).
\end{align*}
This is invertible as soon as $\kappa>0$, with corresponding Green's function
\begin{align}\label{Toda G0}
G_0(n,m;\kappa) = \tfrac1{\sinh\kappa}e^{-\kappa|n - m|}.
\end{align}

For any finite energy state, the operators $L_\pm$ are invertible provided one takes $\kappa$ sufficiently large (depending on the energy).  We write
$$
G_\pm (n,m;\kappa) = \langle\delta_n , L_\pm ^{-1} \delta_m\rangle 
$$
for the corresponding Green's function.  Then,
\begin{equation}\label{Toda naive G}
\tr \bigl\{ L_\pm^{-1}  -  L_0^{-1} \bigr\} = \sum_{n} \bigl[ G_\pm(n,n;\kappa) - \tfrac{1}{\sinh(\kappa)} \bigr].
\end{equation}

It is not difficult to verify that the term in square brackets in \eqref{Toda naive G} is a conserved density for the Toda evolution.   In fact, \eqref{Toda LaxE} yields
\begin{equation}\label{Toda dtG}
\begin{aligned}
\pm \tfrac{d}{dt} G_\pm (n,m) &= a_nG_\pm(n+1,m) - a_{n-1}G_\pm(n-1,m) \\
	& \quad + a_mG_\pm(n,m+1) - a_{m-1}G_\pm(n,m-1).
\end{aligned}
\end{equation}

While \eqref{Toda naive G} is readily verified under the conditions \eqref{Toda l1}, neither side of this equation makes sense for all finite-energy solutions.  Renormalization is required!

For this model, it would be a folly to renormalize using the next term of the Neumann series, namely,
\eq{Toda no log}{
\langle\delta_n, L_0^{-1} (L_\pm - L_0) L_0^{-1} \delta_n\rangle =  - \sum_m \Bigl[\bigl[2a_m-1\bigr] \tfrac{e^{-2\kappa|m+\frac12-n|}}{\sinh^2(\kappa)}
	 \pm b_m \tfrac{ e^{-2\kappa|m-n|} }{\sinh^2(\kappa)}\Bigr],
}
because (unlike for KdV) this is not a conserved density.  We propose the remedy of using
\begin{align}\label{Toda gamma}
\gamma_n^\pm := G_\pm(n,n) - \tfrac{1}{\sinh(\kappa)} - \sum_m \Bigl[\log(2a_m)  \tfrac{e^{-2\kappa|m+\frac12-n|}}{\sinh^2(\kappa)} \pm b_m \tfrac{e^{-2\kappa|m-n|}}{\sinh^2(\kappa)} \Bigr],
\end{align}
inspired by \eqref{Toda Casimirs}.  As we will see shortly, this is a conserved density and the corresponding currents are
\begin{gather}\label{Toda gamma j}
\begin{aligned}
{}^\gamma\!j_n^\pm := 2a_{n-1} &G_\pm(n,n-1)  - \tfrac{e^{-\kappa}}{\sinh(\kappa)}  \\
& - \sum_m\Bigl[ [4a_{m-1}^2-1] \tfrac{e^{-2\kappa|m-n|}}{2\sinh^2(\kappa)} \pm b_m \tfrac{e^{-2\kappa|m+\frac12-n|}}{\sinh^2(\kappa)}\Bigr].
\end{aligned}
\end{gather}

As in all sections of this paper, the principal question to be addressed is not that of finding the right microscopic representation of \eqref{conserve tr}, but rather of finding such a representation of \eqref{conserve det}.  Our answer rests on the following quantities:
\begin{gather}\label{Toda rho}
\begin{aligned}
\rho^\pm_n&:= \kappa - \tfrac12\log\Bigl[1 + \tfrac1{a_nG_\pm(n,n+1)}\Bigr]  \\
&\qquad\qquad\qquad - \sum\limits_{m}\Bigl[\log(2a_m)e^{-2\kappa|n-m|} \pm b_me^{-2\kappa|n + \frac12 - m|}\Bigr]
\end{aligned}
\end{gather}
and
\begin{gather}
\label{Toda j}
 j_n^\pm := \sinh\kappa - \tfrac 1{G_\pm(n,n)} - \sum\limits_{m}\Bigl[[4a_{m-1}^2-1]\tfrac{e^{-2\kappa|n - \frac12 - m|}}2 \pm b_me^{-2\kappa|n-m|}\Bigr].
\end{gather}

\begin{thrm}\label{T:Toda}
There exists $\delta>0$ so that for every $\kappa\geq 1$ and every collection of parameters $(a_m,b_m)\in B^\kappa_\delta$, the sequences $\rho_n^\pm$ and $\gamma^\pm_n$  are non-negative, $\ell^1$, and represent conserved densities for the Toda Lattice.  Concretely,
\begin{align}\label{Toda currents}
\pm \tfrac d{dt} \rho_n^\pm = j^\pm_{n+1} - j^\pm_n \qtq{and}  \pm \tfrac d{dt}\gamma^\pm_n = {}^\gamma\!j^\pm_{n+1} - {}^\gamma\!j^\pm_n,
\end{align}
where \(j_n^\pm,{}^\gamma\!j_n^\pm\in \ell^1\).

Further, for each fixed $n$, $\rho_n^\pm$ and $\gamma^\pm_n$ are convex functions of $(\log2a_m,b_m)$.

Lastly, these yield macroscopic conservation laws of the sought-after form:  If the parameters also satisfy \eqref{Toda l1}, then 
\begin{gather}\label{Toda conserve det}
\sum_n \rho_n^\pm = - \log \det\bigl[ L_\pm / L_0 \bigr] \pm \tfrac{1}{2\sinh(\kappa)} P + \tfrac{ e^{-\kappa} }{2\sinh(\kappa)}  M, \\
\label{Toda conserve tr}
\sum_n \gamma_n^\pm = \tr \bigl\{ L_\pm^{-1}  -  L_0^{-1} \bigr\} \pm \tfrac{\cosh(\kappa)}{2\sinh^3(\kappa)} P + \tfrac{1}{2\sinh^3(\kappa)}  M ,
\end{gather}
where $P$ and $M$ are the Casimirs written in \eqref{Toda Casimirs}.
\end{thrm}

The proof of this theorem will be given at the end of this section, building on a series of preliminary lemmas.  Our first lemma addresses the existence and properties of the Green's functions $G_\pm$:

\begin{lem}\label{L:Toda G}
There exists $\delta>0$ so that the following hold for all $\kappa\geq 1$:  The Lax operators $L_\pm$ are invertible and their Green's functions are given by the series
\begin{align}\label{Toda G series}
G_\pm(n,m;\kappa) - G_0(n,m;\kappa) = \sum_{\ell\geq 1} (-1)^\ell \bigl\langle\delta_n,\ L_0^{-1}\bigl[(L_\pm-L_0)L_0^{-1}\bigr]^\ell \delta_m\bigr\rangle,
\end{align}
which converges uniformly in Hilbert--Schmidt class throughout $B^\kappa_\delta$.  The Green's functions are positive, symmetric under $n\leftrightarrow m$, and satisfy the identities
\begin{gather}
G_\pm(n,n+1)\bigl[1 + a_nG_\pm(n,n+1)\bigr] = a_nG_\pm(n,n)G_\pm(n+1,n+1),\label{QuadraticID}\\
\tfrac{G_\pm(n+1,k)}{G_\pm(n+1,n)} = \tfrac{G_\pm(n,k)}{G_\pm(n,n)}\Bigl[1 + \tfrac1{a_nG_\pm(n,n+1)}\bbo_{k>n}\Bigr],\label{D1}\\
\tfrac{G_\pm(n,k)}{G_\pm(n,n+1)} = \tfrac{G_\pm(n+1,k)}{G_\pm(n+1,n+1)}\Bigl[1 + \tfrac1{a_nG_\pm(n,n+1)}\bbo_{k\leq n}\Bigr],\label{U1}\\
G_\pm(k,\ell) = \tfrac{G_\pm(k,n+1)G_\pm(n,\ell)}{G_\pm(n+1,n)}\qtq{if}k\leq n<\ell.\label{Middle}
\end{gather}
\end{lem}
\begin{proof}
Fourier analysis easily yields the operator norm
$$
\| L_0^{-1} \|_\op = \tfrac{1}{\cosh(\kappa) - 1}.
$$
Elementary manipulations show that, for parameters in $B^\kappa_\delta$ with $\kappa\geq1$ and $\delta>0$ small enough, we have the Hilbert--Schmidt estimate
$$
 \| L_\pm - L_0 \|_{\I_2}^2 \leq \|b_m\|_{\ell^2}^2 + \|2a_m-1\|_{\ell^2}^2 \leq H,
$$
and so
\begin{align}\label{Toda HS small}
\| [L_\pm - L_0]L_0^{-1} \|_{\I_2} \lesssim \delta \sqrt{\kappa} \,e^{ -\kappa} \lesssim \delta.
\end{align}
This settles the convergence of the series \eqref{Toda G series}. The positivity of the Green's function follows from a maximum-principle argument:  Shrinking $\delta$, if necessary, we may ensure that
$$
\cosh(\kappa) > 1 + \| b_m \|_{\ell^\infty} + \| 2a_m-1 \|_{\ell^\infty}
$$
throughout $B^\kappa_\delta$.  The defining property of the Green's function,
$$
[\cosh(\kappa) \mp b_m] G_\pm(m,n) = a_m G_\pm(m+1,n) + a_{m-1} G_\pm(m-1,n) + \delta_{mn},
$$
then shows that for each fixed $n$, if the sequence $m\mapsto G_\pm(m,n)$ achieves a minimum, it must be positive.  However, the series representation shows $G_\pm(m,n)\to 0$ as $m\to\pm\infty$.  We conclude that the infimum of $G_\pm(n,m)$ over all choices of $n,m$ is zero and that this value is never achieved.

The $n\leftrightarrow m$ symmetry of $G_\pm$ is inherited from the self-adjointness of $L_\pm$.

We turn now to the identities \eqref{QuadraticID}, \eqref{D1}, and \eqref{U1}. We remark that these identities can be obtained by writing the Green's function in terms of the Jost solutions. However, we choose to prove them using elementary identities involving the Green's function.  Fixing \(n,k\in\Z\), let us define
\[
I_m := a_m\bigl[G_\pm(n,m+1)G_\pm(m,k) - G_\pm(n,m)G_\pm(m+1,k)\bigr].
\]
Note $I_m\to 0$ as $m\to\pm\infty$.  As the Green's function inverts $L_\pm$, so we have 
\[
I_m - I_{m-1} = (\delta_{mk} - \delta_{nm})G_\pm(n,k).
\]
This may then be summed to obtain
\[
I_m = \begin{cases}
 - G_\pm(n,k)&\qtq{if}n\leq m<k,\\
G_\pm(n,k)&\qtq{if}k\leq m<n,\\
0&\qtq{otherwise.}
\end{cases}
\]
The identity \eqref{QuadraticID} then follows by taking \(k = n+1\) and \(m=n\). Similarly, the identity \eqref{D1} follows from taking \(m = n\), whereas the identity \eqref{U1} from taking \(m=n-1\) and then replacing \(n\) by \(n+1\). Finally, in the case that \(k\leq n<\ell\) we may write our identity for \(I_m\) in the form
\[
\tfrac{G(k,n+1)}{G(n,n+1)} = \tfrac{G(k,n+2)}{G(n,n+2)} = \dots = \tfrac{G(k,\ell)}{G(n,\ell)},
\]
which yields \eqref{Middle}.
\end{proof}

The identity \eqref{QuadraticID} provides an alternate expression for $\rho_n$ via the identities
\begin{align}\label{alt rho}
\tfrac12\log\Bigl[1 + \tfrac1{a_nG_\pm(n,n+1)}\Bigr] &= \arcsinh\Bigl[\tfrac1{\sqrt{4a_n^2G_\pm(n,n)G_\pm(n+1,n+1)}}\Bigr] \\
&=\tfrac12\log\Bigl[\tfrac{G_\pm(n,n)G_\pm(n+1,n+1)}{G_\pm(n,n+1)^2}\Bigr]. \label{alt rho'}
\end{align}
Each resulting form of $\rho_n$ has its own merits. The original expression is the most compact, whereas RHS\eqref{alt rho} provides the strongest link to the reciprocal of the diagonal Green's function in the continuum limit. However, it is RHS\eqref{alt rho'} that will dominate our proof of Theorem~\ref{T:Toda}.

\begin{proof}[Proof of Theorem~\ref{T:Toda}]
Lemma~\ref{L:Toda G} already guarantees that $\rho^\pm_n$ and $\gamma^\pm_n$ (as well as their purported currents) are well-defined, provided $\delta>0$ is chosen sufficiently small.

Using \eqref{alt rho'} we may write \(\rho_n\) in the form
\[
\rho_n^\pm = \kappa - \tfrac12\log\Biggl[\tfrac{G_\pm(n,n)G_\pm(n+1,n+1)}{G_\pm(n,n+1)^2}\Biggr] - \sum\limits_{m}\Bigl[\log(2a_m)e^{-2\kappa|n-m|} \pm b_me^{-2\kappa|n + \frac12 - m|}\Bigr].
\]
The identities \eqref{Toda currents} then follow from \eqref{Toda PDE}, \eqref{Toda dtG}, \eqref{D1}, and \eqref{U1}.

We now turn our attention to the identities \eqref{Toda conserve det} and \eqref{Toda conserve tr}.  Here we assume that the parameters $(a_m, b_m)$ satisfy \eqref{Toda l1}, which ensures that \(L^{-1} - L_0^{-1}\) is trace class. The identity \eqref{Toda conserve tr} follows directly from the definition \eqref{Toda gamma}. To justify \eqref{Toda conserve det}, we apply the resolvent identity combined with \eqref{D1} and \eqref{U1} to write
\begin{align*}
&\tfrac1{\sinh\kappa}\p_\kappa\rho_n^\pm\\
&\qquad= \tfrac1{\sinh\kappa} - \sum\limits_m\Bigl[ \tfrac{G_\pm(n+1,m)G_\pm(m,n)}{G_\pm(n+1,n)} - \tfrac{G_\pm(n,m)G_\pm(m,n)}{2G_\pm(n,n)} - \tfrac{G_\pm(n+1,m)G_\pm(m,n+1)}{2G_\pm(n+1,n+1)}\Bigr]\\
&\qquad\quad + \sum\limits_{m}\Bigl[\log(2a_m)\tfrac{2|n-m|e^{-2\kappa|n-m|}}{\sinh\kappa} \pm b_m\tfrac{2|n + \frac12 - m|e^{-2\kappa|n + \frac12 - m|}}{\sinh\kappa}\Bigr]\\
&\qquad= - \gamma_n^\pm + \sigma_{n+1}^\pm - \sigma_n^\pm,
\end{align*}
where we take
\begin{align*}
\sigma_n^\pm &= \tfrac1{2\sinh\kappa} + \tfrac12\sum_m \sgn(n - m - \tfrac12)\tfrac{G_\pm(n,m)G_\pm(m,n)}{G_\pm(n,n)}\\
&\quad - \sum\limits_m \Bigl[\log(2a_m) \tfrac{(n - 1 - m)e^{-2\kappa|n - m - \frac12|}}{\sinh^2(\kappa)} \pm b_m\tfrac{(n-\frac12 - m)e^{-2\kappa|n - m|}}{\sinh^2(\kappa)}\Bigr].
\end{align*}
Note that \(\sigma_n^\pm\) is well-defined and vanishes as \(n\to\infty\) thanks to Lemma~\ref{L:Toda G}. Under the assumption \eqref{Toda l1}, we have \(\sigma_n^\pm\in\ell^1\) and hence
\begin{align}\label{rho deriv}
\tfrac1{\sinh\kappa}\partial_\kappa\sum_n\rho_n^\pm = - \sum_n\gamma_n^\pm.
\end{align}
On the other hand, \eqref{Toda HS small} implies that the expression
\begin{align}
F_\pm:=  - \log \det\bigl[ L_\pm / L_0 \bigr] =\sum_{\ell\geq 1} \tfrac{1}{\ell} (-1)^\ell \tr\Bigl\{ \big([L_\pm - L_0]L_0^{-1}\bigr)^{\ell} \Bigr\}
\end{align}
converges uniformly for $\kappa\geq 1$ and parameters $(a_m,b_m)\in B^\kappa_\delta$ satisfying \eqref{Toda l1}. When taking the derivative with respect to the spectral parameter, we get
$$
\tfrac1{\sinh\kappa}\partial_\kappa F_\pm=\sum_{\ell\geq 1}  (-1)^{\ell+1} \tr\Bigl\{L_0^{-1} \big( [L_\pm - L_0] L_0^{-1}\bigr)^{\ell} \Bigr\}
	= - \tr\Bigl\{L_\pm^{-1} - L_0^{-1} \Bigr\}.
$$
Subtracting this from \eqref{rho deriv} and integrating with respect to the spectral parameter to infinity, we obtain \eqref{Toda conserve det}.

Next, we turn to the issue of convexity. This trivializes verifying non-negativity, which in turn aids in the proof of summability. We will prove convexity via the second derivative test.  (In view of \eqref{Toda G series},  differentiability is not an issue.) Given a direction $(c_m,d_m) \in\ell^2(\Z)\times\ell^2(\Z)$ and a function \(F\) on $B_\delta^\kappa$, we define
\[
D^kF = \tfrac{d^k}{ds^k}\Big|_{s=0} F\bigl(a_m e^{sc_m},b_m + sd_m\bigr).
\]
Our goal is to show $D^2 \gamma^\pm_n\geq 0$ and $D^2 \rho^\pm_n\geq 0$.  Note that the renormalization terms, appearing as sums over $m$ in \eqref{Toda gamma} and \eqref{Toda rho}, are linear in $\log(2a_m)$ and in $b_m$ and so these will not affect the convexity computations.  

Let us begin with $\gamma_n$; we will suspend the $\pm$ notations since the distinction plays virtually no role in the computations that follow. For concreteness, we consider the \(+\) case. As a preliminary, we compute
\begin{align}\label{D2 Gnm}
D^2G(n,m) = 2\<\delta_n,L^{-1}(D L)L^{-1}(DL)L^{-1}\delta_m\> - \<\delta_n,L^{-1}(D^2L)L^{-1}\delta_m\>,
\end{align}
using the resolvent identity.  The derivatives of the operator $L$ are given by
\begin{align*}
(DLf)_n &= -a_n c_n f_{n+1} - a_{n-1} c_{n-1}f_{n-1} - d_nf_n, \\
(D^2Lf)_n &= -a_n c_n^2f_{n+1} - a_{n-1} c_{n-1}^2 f_{n-1}.
\end{align*}

Note that the first term in \eqref{D2 Gnm} is positive when $m=n$ because $L^{-1}$ is positive definite.  Regarding the second term in \eqref{D2 Gnm}, we have
\begin{align}\label{-GD2LG +ve}
- \<\delta_n,L^{-1}(D^2L)L^{-1}\delta_m\> \geq 0 \quad \text{for all $m,n\in \Z$.}
\end{align}
This is because the Green's function is positive and the individual matrix entries of $D^2 L$ are less than or equal to zero.  Thus $\gamma_n$ is convex.  This convexity also guarantees the non-negativity of $\gamma_n$ because
$$
\gamma_n =0 \qtq{and} D \gamma_n = 0 \quad\text{when $a_m \equiv \tfrac12$ and $b_m\equiv 0$,}
$$
irrespective of the direction in which the derivative is taken.

Let us begin our discussion of $\rho_n$ in the same place.  Direct computation shows
\begin{align}\label{rho values}
\rho_n =0 \qtq{and} D \rho_n = 0 \quad\text{when $a_m \equiv \tfrac12$ and $b_m\equiv 0$,}
\end{align}
irrespective of the direction in which the derivative is taken.  One other first derivative computation is important.  Using \eqref{QuadraticID} we obtain
\begin{align*}
\tfrac{\partial\ }{\partial a_n} \rho_n = \frac{\frac{\partial\ }{\partial a_n} [a_nG(n,n+1)]}{2a_nG(n,n+1)[1 + a_nG(n,n+1)]} - \frac{1}{a_n} =0.
\end{align*}
This shows that $\rho_n$ fails to be strictly convex.  More constructively, it shows that all terms in $D^2 \rho_n$ involving $c_n$ vanish.  Thus we may assume that $c_n=0$ henceforth.  This will simplify matters considerably when we invoke it later.  For the moment however, we use \eqref{alt rho'} to compute
\begin{equation}\begin{aligned}
D^2\rho_n &= - \Bigl[\tfrac{DG(n+1,n)}{G(n+1,n)}\Bigr]^2 + \tfrac12\Bigl[\tfrac{DG(n,n)}{G(n,n)}\Bigr]^2 + \tfrac12\Bigl[\tfrac{DG(n+1,n+1)}{G(n+1,n+1)}\Bigr]^2 \\
&\quad\ {}+ \tfrac{D^2G(n+1,n)}{G(n+1,n)} - \tfrac{D^2G(n,n)}{2G(n,n)}  - \tfrac{D^2G(n+1,n+1)}{2G(n+1,n+1)} .
\end{aligned}\end{equation}

The next step is to expand out the second line here using \eqref{D2 Gnm}.  Rather than writing down the result immediately, let us focus on the resulting terms that involve $D^2L$ (as opposed to those quadratic in $DL$):
\begin{align*}
&{} - \tfrac{\<\delta_{n+1},L^{-1}(D^2L)L^{-1}\delta_n\>}{G(n+1,n)} + \tfrac{\<\delta_n,L^{-1}(D^2L)L^{-1}\delta_n\>}{2G(n,n)} + \tfrac{\<\delta_{n+1},L^{-1}(D^2L)L^{-1}\delta_{n+1}\>}{2G(n+1,n+1)}\\
&\qquad= - \tfrac12\sum\limits_{m,k} \<\delta_m,(D^2L)\delta_k\>G(k,n)\Bigl[\tfrac{G(n+1,m)}{G(n+1,n)} - \tfrac{G(n,m)}{G(n,n)}\Bigr]\\
&\qquad\quad - \tfrac12\sum\limits_{m,k} G(n+1,m)\<\delta_m,(D^2L)\delta_k\>\Bigl[\tfrac{G(k,n)}{G(n+1,n)} - \tfrac{G(k,n+1)}{G(n+1,n+1)}\Bigr].
\end{align*}
By \eqref{D1} and \eqref{U1}, each quantity in square brackets here is non-negative.  Combining this with the fact that every entry in the matrix $D^2L$ is less than or equal to zero, we see that this whole expression is non-negative.  Thus,
\begin{align}
D^2\rho_n &\geq \sum_{m,k,r,\ell}\<\delta_m,(DL)\delta_k\>\<\delta_\ell,(DL)\delta_r\>K_n(m,k,\ell,r)\label{D2''},
\end{align}
where
\begin{align*}
K_n(m,k,\ell,r) &= \tfrac{G(n+1,m)G(k,\ell)G(r,n)}{G(n+1,n)} + \tfrac{G(n,m)G(k,\ell)G(r,n+1)}{G(n+1,n)}\\
&\quad - \tfrac{G(n+1,m)G(k,n)G(n+1,\ell)G(r,n)}{2G(n+1,n)^2} - \tfrac{G(n,m)G(k,n+1)G(n,\ell)G(r,n+1)}{2G(n+1,n)^2}\\
&\quad - \tfrac{G(n,m)G(k,\ell)G(r,n)}{G(n,n)} + \tfrac{G(n,m)G(k,n)G(n,\ell)G(r,n)}{2G(n,n)^2}\\
&\quad - \tfrac{G(n+1,m)G(k,\ell)G(r,n+1)}{G(n+1,n+1)}  + \tfrac{G(n+1,m)G(k,n+1)G(n+1,\ell)G(r,n+1)}{2G(n+1,n+1)^2}.
\end{align*}

Due to the vanishing of $c_n$, we see that $\<\delta_m,(DL)\delta_k\>$ vanishes unless $m,k\leq n$ or $m,k>n$.  Similarly, $\<\delta_\ell,(DL)\delta_r\>$ vanishes unless $\ell$ and $r$ lie on the same side of $n$.  This allows us to restrict our attention to just three cases, in each of which we may then simplify the formula for $K_n$ by exploiting \eqref{D1}, \eqref{U1}, and \eqref{Middle}:
\begin{align*}
K_n &\!=\!0,\quad \text{if $m,k\leq n < \ell,r$ or $\ell,r\leq n < m,k$,} \\
K_n &\!=\! \Bigl[\tfrac{G(k,\ell)}{G(n,n)} - \tfrac{G(k,n)G(n,\ell)}{G(n,n)^2}\Bigr]G(n,m)G(r,n)\Bigl[1 - \tfrac{G(n+1,n)^2}{G(n,n)G(n+1,n+1)}\Bigr]\\
&\quad + \tfrac{G(n,m)G(k,n)G(n,\ell)G(r,n)}{2G(n,n)^2}\Bigl[1 - \tfrac{G(n+1,n)^2}{G(n,n)G(n+1,n+1)}\Bigr]^2, \quad \text{if $m,k,\ell,r \leq n$,}\\
K_n &\!=\! \Bigl[\tfrac{G(k,\ell)}{G(n+1,n+1)}\! - \!\tfrac{G(k,n+1)G(n+1,\ell)}{G(n+1,n+1)^2}\Bigr]G(n+1,m)G(r,n+1) \Bigl[1\! -\! \tfrac{G(n+1,n)^2}{G(n,n)G(n+1,n+1)}\Bigr]\\
&\quad + \tfrac{G(n+1,m)G(k,n+1)G(n+1,\ell)G(r,n+1)}{2G(n+1,n+1)^2}\Bigl[1 - \tfrac{G(n+1,n)^2}{G(n,n)G(n+1,n+1)}\Bigr]^2, \; \text{if $m,k,\ell,r > n$.}
\end{align*}

Using \eqref{D2''}, we may now see that $D^2\rho_n \geq 0$.  Let us explain here why the sum over $m,k,\ell,r \leq n$ is non-negative; the argument in the case $m,k,\ell,r > n$ is analogous.
When $m,k,\ell,r \leq n$, the contribution of the second line in the expression of $K_n$ is given by
$$
\tfrac1{2G(n,n)^2}\Bigl[1 - \tfrac{G(n+1,n)^2}{G(n,n)G(n+1,n+1)}\Bigr]^2 \Bigl\{ \sum_{m,k\leq n} G(n,m) G(k,n) \langle \delta_m,  (DL)\delta_k\rangle \Bigr\}^2\geq 0.
$$
We turn now to the contribution of the first line in the expression of $K_n$.  Using \eqref{QuadraticID}, we observe that
\[
1 - \tfrac{G(n+1,n)^2}{G(n,n)G(n+1,n+1)} = \tfrac1{1 + a_nG(n,n+1)}\geq 0.
\]
Moreover, writing
$$
\psi_k:= \sum_{m\leq n} G(n,m)  \langle \delta_m,  (DL)\delta_k\rangle,
$$
and recalling that \(c_n = 0\), we find
\begin{align*}
\sum_{m,k,\ell,r\leq n}&\Bigl[\tfrac{G(k,\ell)}{G(n,n)} - \tfrac{G(k,n)G(n,\ell)}{G(n,n)^2}\Bigr]G(n,m)G(r,n)\langle \delta_m,  (DL)\delta_k\rangle\langle \delta_\ell,  (DL)\delta_r\rangle\\
&= \tfrac1{G(n,n)^2} \langle \psi, L^{-1} \psi\rangle \langle \delta_n, L^{-1}\delta_n\rangle - \langle \delta_n, L^{-1} \psi\rangle^2\geq 0, 
\end{align*}
where the inequality follows from the fact that $L^{-1}$ is positive definite and an application of Cauchy--Schwarz.  This completes the proof of the convexity of $\rho_n$.

The non-negativity of $\rho_n$ follows from its convexity and \eqref{rho values}.

Finally, we turn to the issue of summability. For sequences satisfying \eqref{Toda l1}, we use \eqref{Toda no log} to write the identities \eqref{Toda conserve det}, \eqref{Toda conserve tr} as
\begin{align*}
\sum_n \rho_n^\pm &= \sum_{\ell\geq 2} \tfrac{1}{\ell} (-1)^\ell \tr\Bigl\{ \big([L_\pm - L_0]L_0^{-1}\bigr)^{\ell} \Bigr\} + \tfrac{ e^{-\kappa} }{\sinh(\kappa)} \sum_m V\bigl(-\log(2a_m)\bigr),\\
\sum_n \gamma_n^\pm &= \sum_{\ell\geq 2} (-1)^\ell \tr\Bigl\{L_0^{-1}\big([L_\pm - L_0]L_0^{-1}\bigr)^{\ell} \Bigr\} + \tfrac1{\sinh^3(\kappa)} \sum_m V\bigl(-\log(2a_m)\bigr).
\end{align*}
From \eqref{Toda HS small}, the right-hand sides converge for parameters $(a_m,b_m)\in B^\kappa_\delta$. That the sequences $\rho_n^\pm,\gamma_n^\pm\in\ell^1$ for parameters $(a_m,b_m)\in B^\kappa_\delta$ then follows from the positivity of $\rho_n^{\pm},\gamma_n^\pm$, approximating the parameters by sequences satisfying \eqref{Toda l1}.

Turning to the summability of the currents, it is evident from the resolvent expansion \eqref{Toda G series} and the estimate \eqref{Toda HS small} that \({}^\gamma\!j_n^\pm\in\ell^1\). For \(j_n\), first note that \eqref{Toda HS small} ensures that \(G_{\pm}(n,n)\gtrsim 1\), uniformly in \(n\), and consequently that
\[
\sinh(\kappa) + \sinh^2(\kappa)\langle\delta_n, L_0^{-1} (L_\pm - L_0) L_0^{-1} \delta_n\rangle - \tfrac1{G_\pm(n,n)} \in \ell^1.
\]
That \(j_n\in \ell^1\) then follows from \eqref{Toda no log}.
\end{proof}

\section{Ablowitz--Ladik}\label{S:AL} 

The Ablowitz--Ladik system \cite{AL1,AL2} is an integrable discrete form of the cubic non-linear Schr\"odinger equation.  It comes in two flavors, which we may write together via the expedient of defining $\beta_n:=\bar \alpha_n$ in the defocusing case and $\beta_n:=-\bar \alpha_n$ in the focusing case:
\begin{align}\label{AL eqn}\tag{AL}
i \partial_t \alpha_n = 2\alpha_n - (1-\alpha_n\beta_n)(\alpha_{n+1} + \alpha_{n-1}).
\end{align}
In the defocusing case, it is required that all $\alpha_n\in\D$, the open unit disk in $\C$; as we will see below, this property is preserved by the flow.  In the focusing case, $\alpha_n\in\C$ are unrestricted.

It is a trivial matter to see that these flows are locally well-posed on $\ell^2$.  This extends to global well-posedness via the conservation of
\[
M := -\sum_{n\in \Z}\log(1 - \alpha_n\beta_n).
\]
This conservation law also guarantees that the restriction $|\alpha_n|<1$ is preserved by the defocusing flow.  Another important conserved quantity is
\[
H := \sum_{n\in \Z}\Bigl(-\alpha_n\beta_{n+1} - \alpha_{n+1}\beta_{n} - 2\log(1 - \alpha_n\beta_n)\Bigr),
\]
which serves as the Hamiltonian for \eqref{AL eqn} with respect to the Poisson structure
\[
\{F,G\} := \tfrac1i\sum\limits_{n\in \Z}(1 - \alpha_n\beta_n)\Bigl[\tfrac{\p F}{\p \alpha_n}\tfrac{\p G}{\p \beta_n} - \tfrac{\p F}{\p \beta_n}\tfrac{\p G}{\p\alpha_n}\Bigr].
\]

The evolutions \eqref{AL eqn} admit a zero curvature representation based on the matrices
\begin{align*}
U_n(z) := \begin{bmatrix}z & \alpha_n \\ \beta_n & z^{-1} \end{bmatrix}
	\ \text{ and }\ 
V_n(z) := i \begin{bmatrix} z^2 - 1 - \alpha_n\beta_{n-1} & z\alpha_{n}-z^{-1}\alpha_{n-1} \\ z\beta_{n-1}-z^{-1}\beta_{n} & 1+\alpha_{n-1}\beta_n-z^{-2} \end{bmatrix}.
\end{align*}
Concretely, \eqref{AL eqn} is equivalent to $\partial_t U_n = V_{n+1} U_n - U_n V_n$.  This representation can be reorganized into one resembling \eqref{E:LP} in several different ways.  Although it is possible to give such a Lax representation where the spectral parameter appears in the classical way (see \cite{Nenciu}), we choose a different path that is more conducive to drawing analogies with Section~\ref{S:NLS}.

Let $\mathbf U$ and $\mathbf V$ denote the operators on $\ell^2(\Z)\otimes \C^2$ defined by applying $U_n$ and $V_n$ at each lattice site.  Specifically,
\begin{align*}
\biggl(\mathbf U \begin{bmatrix}a  \\ b \end{bmatrix}\biggr)_n = U_n\begin{bmatrix}a_n  \\ b_n \end{bmatrix}.
\end{align*}
We also define $S$ to be the scalar (left-)shift operator $(Sf)_n = f_{n+1}$ and $\mathbf{S}=S\otimes I$ as the vector analogue on $\ell^2(\Z)\otimes \C^2$:
\begin{align*}
\biggl(\mathbf S \begin{bmatrix}a  \\ b \end{bmatrix}\biggr)_n = \begin{bmatrix}a_{n+1}  \\ b_{n+1} \end{bmatrix}.
\end{align*}
Writing $\mathbf I =I\otimes I$, it is elementary to see that 
\begin{align}\label{E:my LP}
\partial_t (\mathbf S^{-1}\mathbf  U - \mathbf I ) = [\mathbf  V ,\ \mathbf S^{-1} \mathbf U - \mathbf I ]
\end{align}
is equivalent to the zero curvature condition stated earlier.  Nevertheless, we do not choose $\mathbf S^{-1} \mathbf U - \mathbf I$ as our basic object, but rather
\begin{align}\label{AL L}
\mathbf L(z;\alpha) := \mathbf U - \mathbf S = \begin{bmatrix}z - S & \alpha \\ \beta & z^{-1}- S\end{bmatrix} .
\end{align}
In the special case $\alpha\equiv0$, we write $\mathbf U_0$ and  $\mathbf L_0$.

Our naive expectations \eqref{conserve det} and \eqref{conserve tr} were expressed in terms of the operator appearing in \eqref{E:LP}.   Our analogue, \eqref{E:my LP}, involves $\mathbf S^{-1}\mathbf  U - \mathbf I$ rather than $\mathbf L$.   From the simple relation between the two, we find that our original predictions are equivalent to 
\begin{gather}
\partial_t \log \det \bigl[ \mathbf L \, \mathbf L_0^{-1} \bigr] = 0,\label{AL det}
\\
\partial_t \tr \bigl\{ \bigl(\mathbf L^{-1} - \mathbf L_0^{-1}\bigr) \mathbf S \bigr\} =0.\label{AL tr prediction}
\end{gather}
However, the second of these is vacuous. Indeed, the identity \eqref{trID} shows that
\begin{equation}\label{Vacuous}
\tr \bigl\{ \bigl(\mathbf L^{-1} - \mathbf L_0^{-1}\bigr) \mathbf S \bigr\} = 0.
\end{equation}
Instead, we employ the Pauli matrix \(\sigma_3 = \bigl[\!\begin{smallmatrix} 1 & 0 \\ 0 & -1\end{smallmatrix}\!\bigr]\) and replace \eqref{AL tr prediction} by
\begin{equation}\label{AL tr}
\partial_t \tr \bigl\{ \bigl(\mathbf L^{-1} - \mathbf L_0^{-1}\bigr) \mathbf S\sigma_3 \bigr\} = 0.
\end{equation}
The introduction of $\sigma_3$ is easily intuited by comparing the continuum limit of $\mathbf L$ to the Lax operator used in Section~\ref{S:NLS}: there is sign flip in the bottom row.  The presence of the shift operator cannot be explained by such naive reasoning --- it drops out in the continuum limit.  This underlines the point made in the introduction that the discrete laws discussed here are more fundamental than their continuum analogues.

To describe our microscopic conservation laws, we first introduce the matrix-valued Green's function \(G(n,m;z)\) for \(\mb L^{-1}\), which will shortly be shown to be well-defined whenever \(|z|>1\) is sufficiently large relative to the mass. Further, from \eqref{E:my LP} we obtain
\begin{equation}\label{AL dtG}
\p_tG(n,m) = V_nG(n,m) - G(n,m)V_{m+1}.
\end{equation}
We then introduce the densities attendant to the conservation laws \eqref{AL det} and \eqref{AL tr}
\begin{align}
\rho_n &:= \tfrac12\log\Bigl[1 + \tfrac{\alpha_nG_{21}(n,n)}{zG_{11}(n,n)}\Bigr] + \tfrac12\log\Bigl[1 + \tfrac{\beta_{n+1}G_{12}(n+1,n+1)}{zG_{11}(n+1,n+1)}\Bigr],\\
\gamma_n &:= G_{11}(n+1,n) - G_{22}(n+1,n) - 1,
\end{align}
and the corresponding currents
\begin{align}
j_n &:= \tfrac i2\bigl[z\alpha_n - z^{-1}\alpha_{n-1}\bigr]\tfrac{G_{21}(n,n)}{G_{11}(n,n)} + \tfrac i2\bigl[z\beta_n - z^{-1}\beta_{n+1}\bigr]\tfrac{G_{12}(n,n)}{G_{11}(n,n)}\\
&\quad  - \tfrac i2\alpha_n\beta_{n-1} - \tfrac i2\alpha_{n+1}\beta_n,\notag\\
{}^\gamma\!j_n &:= 2izG_{11}(n+1,n-1) +2iz^{-1}\bigl[G_{22}(n+1,n-1)+z^{-1}\Bigr].
\end{align}

Analogously to the Toda Lattice, for \(|z|>1\) we introduce the ball
\[
B_\delta^z = \Bigl\{\alpha_m:|M|e^{|M|} < \delta^2 \tfrac{|z|^2 - 1}{|z|}\Bigr\}.
\]
Our main result is then the following:

\begin{thrm}\label{T:AL}
There exists $\delta>0$ so that for every $|z|\geq 2$ and every $\alpha_m\in B^z_\delta$, the sequences $\rho_n,\gamma_n\in\ell^1$ represent conserved densities for the Ablowitz--Ladik equation.  Concretely,
\begin{align}\label{AL currents}
\partial_t \rho_n = j_{n+1} - j_n \qtq{and}  \partial_t \gamma_n = {}^\gamma\!j_{n+1} - {}^\gamma\!j_n,
\end{align}
where \(j_n,{}^\gamma\!j_n\in \ell^1\).

Further, these yield macroscopic conservation laws of the sought-after form:
\begin{gather}\label{AL conserve det}
\sum_n \rho_n = \log \det \bigl[ \mathbf L \, \mathbf L_0^{-1} \bigr], \\
\label{AL conserve tr}
\sum_n \gamma_n = \tr \bigl\{ \bigl(\mathbf L^{-1} - \mathbf L_0^{-1}\bigr) \mathbf S\sigma_3 \bigr\}.
\end{gather}
\end{thrm}

We begin our analysis with some basic properties of the resolvent of the (scalar) shift operator:

\begin{lem}\label{l:S} For $\alpha_m\in\ell^2(\Z)$ and $|z|> 1$,
\begin{gather*}
\langle \delta_n,\  (z-S)^{-1} \delta_m \rangle = z^{n-m-1}\bbo_{n\leq m} \qtq{and} 
\langle \delta_n,\  (S-z^{-1})^{-1} \delta_m \rangle =z^{m-n+1}\bbo_{n>m}.
\end{gather*}
In particular, the Green's function for \(\mb L_0^{-1}\) is
\begin{equation}\label{G0-AL}
G_0(n,m;z) = \begin{bmatrix} z^{n-m-1}\bbo_{n\leq m} & 0 \\ 0 & -z^{m-n+1}\bbo_{n>m} \end{bmatrix}.
\end{equation}

Moreover, defining operators
\begin{gather}
\Lambda := \alpha (S-z^{-1})^{-1} \qtq{and} \Gamma  := \beta (z-S)^{-1},
\end{gather}
we have the following Hilbert--Schmidt norms:
\begin{gather}\label{AL HS}
\bigl\| \Lambda \bigr\|_{\I_2}^2 = \tfrac{|z|^2}{|z|^2 - 1} \|\alpha\|_{\ell^2}^2
\qtq{and}
\bigl\| \Gamma   \bigr\|_{\I_2}^2  = \tfrac{1}{|z|^2 - 1} \|\alpha\|_{\ell^2}^2.
\end{gather}
\end{lem}

Using the previous lemma, it is easy to construct $\mathbf L^{-1}$ by Neumann series:

\begin{lem}\label{L:AL GF}
There exists \(\delta>0\) so that for every \(|z|\geq 2\) and \(\alpha_m\in B_\delta^z\),
$\mathbf L$ is invertible. The matrix-valued Green's function has entries
\begin{align*}
G_{11}(n,m) &=  \sum_{k\geq 0} (-1)^k \langle \delta_n, (z-S)^{-1} [\Lambda\Gamma]^k \delta_m\rangle, \\
G_{12}(n,m) &=  \sum_{k\geq 0} (-1)^k \langle \delta_n, (z-S)^{-1} \Lambda [\Gamma \Lambda]^k \delta_m\rangle, \\
G_{21}(n,m) &=  \sum_{k\geq 0} (-1)^k \langle \delta_n, (S-z^{-1})^{-1} \Gamma [\Lambda \Gamma]^k \delta_m\rangle,\\
G_{22}(n,m) &= -\sum_{k\geq 0} (-1)^k \langle \delta_n, (S-z^{-1})^{-1} [\Gamma \Lambda]^k \delta_m\rangle.
\end{align*}
Moreover, we have the following identities:
\begin{align}
\det G(n,m) &= 0,\label{detID}\\
\tr G(n+1,n) &= -1.\label{trID}
\end{align}
In addition, we have
\begin{align}
\tfrac{G_{12}(n+1,k)}{1 + G_{11}(n+1,n)} &= \tfrac{G_{12}(n,k)}{G_{11}(n,n)} + \tfrac{\alpha_nG_{22}(n+1,k)\bbo_{k<n}}{(1 - \alpha_n\beta_n)(1 + G_{11}(n+1,n))G_{11}(n,n)}\label{AL nD},\\
\tfrac{G_{21}(k,n)}{1 + G_{11}(n+1,n)} &= \tfrac{G_{21}(k,n+1)}{G_{11}(n+1,n+1)} + \tfrac{ \beta_{n+1}G_{22}(k,n)\bbo_{k>n+1}}{(1 - \alpha_{n+1}\beta_{n+1})(1 + G_{11}(n+1,n))G_{11}(n+1,n+1)}\label{AL nU},\\
\tfrac{\delta_{nk} + G_{11}(n+1,k)}{1 + G_{11}(n+1,n)} &= \tfrac{G_{11}(n,k)}{G_{11}(n,n)} + \tfrac{\alpha_nG_{21}(n+1,k)\bbo_{k<n}}{(1 - \alpha_n\beta_n)G_{11}(n,n)(1 + G_{11}(n+1,n))},\label{AL nD2}\\
\tfrac{\delta_{n+1,k} + G_{11}(k,n)}{1 + G_{11}(n+1,n)} &= \tfrac{G_{11}(k,n+1)}{G_{11}(n+1,n+1)} + \tfrac{\beta_{n+1}G_{12}(k,n)\bbo_{k>n+1}}{(1 - \alpha_{n+1}\beta_{n+1})G_{11}(n+1,n+1)(1 + G_{11}(n+1,n))},\label{AL nU2}
\end{align}
and the denominators of these expressions are non-zero.
\end{lem}

\begin{proof}
The series for the entries of $G(n,m)$ are a simple matter of book-keeping in the Neumann series.
The question of convergence is settled by the fact that
\[
\|\alpha\|_{\ell^2}^2\leq |M|e^{|M|}\leq \tfrac{|z|^2-1}{|z|}\delta^2,
\]
and so from \eqref{AL HS},
\begin{equation}\label{AL HS small}
\| \Lambda \|_{\I_2}^2 \leq |z|\delta^2\qtq{and}  \|\Gamma\|_{\I_2}^2 \leq \tfrac1{|z|}\delta^2.
\end{equation}

Using that
\[
\|(z - S)^{-1}\|_{\op}\leq \tfrac1{|z| - 1}\qtq{and}\|(S - z^{-1})^{-1}\|_{\op}\leq \tfrac{|z|}{|z| - 1},
\]
we have 
\begin{align}
\sum\limits_{n,m}|G_{12}(n,m)|^2 &\lesssim \Bigl\{ \|(z-S)^{-1}\|_{\op}\|\Lambda\|_{\I_2}\sum_{\ell\geq 0}\|\Lambda\|_{\I_2}^\ell\|\Gamma\|_{\I_2}^\ell\Bigr\}^2\lesssim \tfrac{|z|}{(|z|-1)^2}\delta^2,\label{g12}\\
\sum\limits_{n,m}|G_{21}(n,m)|^2 &\lesssim \Bigl\{\|(S-z^{-1})^{-1}\|_{\op}\|\Gamma\|_{\I_2}\sum_{\ell\geq 0}\|\Lambda\|_{\I_2}^\ell\|\Gamma\|_{\I_2}^\ell\Bigr\}^2\lesssim \tfrac{|z|}{(|z|-1)^2}\delta^2,\label{g21}
\end{align}
and recalling \eqref{G0-AL},
\begin{align}
&\sum\limits_{n,m}|G_{11}(n,m) - z^{n-m-1}\bbo_{n\leq m}|^2 \notag\\
&\qquad\qquad \qquad\qquad\lesssim \Bigl\{ \|(z-S)^{-1}\|_{\op}\sum_{\ell\geq 1}\|\Lambda\|_{\I_2}^\ell\|\Gamma\|_{\I_2}^\ell\Bigr\}^2\lesssim \tfrac1{(|z| - 1)^2}\delta^4.\label{g11-HS}
\end{align}

By observing that all but the first terms the Neumann series are products of Hilbert--Schmidt operators,  the Cauchy--Schwarz inequality yields
\begin{align}
&\sup\limits_m\sum\limits_n|G_{11}(n,n+m) - z^{-m-1}\bbo_{m\geq 0}|\notag\\
&\qquad\qquad \qquad\qquad\lesssim \|(z - S)^{-1}\|_{\op}\sum_{\ell\geq 1}\|\Lambda\|_{\I_2}^\ell\|\Gamma\|_{\I_2}^\ell\lesssim \tfrac1{|z| - 1}\delta^2,\label{g11}\\
&\sup\limits_m\sum\limits_n|G_{22}(n,n+m) + z^{m+1}\bbo_{m < 0}|\notag\\
&\qquad\qquad\qquad\qquad\lesssim \|(S - z^{-1})^{-1}\|_{\op}\sum_{\ell\geq 1}\|\Lambda\|_{\I_2}^\ell\|\Gamma\|_{\I_2}^\ell\lesssim \tfrac{|z|}{|z| - 1}\delta^2.\label{g22}
\end{align} 

The remaining identities can again be derived by considering the Jost solutions. However, as in the preceding section, we instead choose to derive these from elementary identities involving the Green's function. We introduce
\begin{align*}
I_n(m,k) &:= G_{11}(n,m)G_{22}(n,k) - G_{12}(n,k)G_{21}(n,m),\\
J_n(m,k) &:= G_{11}(m,n)G_{22}(k,n) - G_{12}(m,n)G_{21}(k,n),
\end{align*}
and note that \(I_n,J_n\to 0\) as \(n\to\pm\infty\). As \(G\) inverts \(\mb L\), we have
\begin{align}
I_{n+1}(m,k) &= (1 - \alpha_n\beta_n)I_n(m,k) - \delta_{nk}G_{11}(n+1,m) - \delta_{nm}G_{22}(n+1,k),\label{In}\\
J_{n-1} (m,k)&= (1 - \alpha_n\beta_n)J_n (m,k)- \delta_{nk}G_{11}(m,n-1) - \delta_{nm}G_{22}(k,n-1).\label{Out}
\end{align}
In particular
\begin{align}
I_n(m,k) = 0\qtq{if}n\leq \min\{m,k\}\qtq{or}n>\max\{m,k\},\label{In0}\\
J_n(m,k) = 0\qtq{if}n<\min\{m,k\}\qtq{or}n\geq \max\{m,k\}.\label{Jn0}
\end{align}

The identity \eqref{detID} follows from taking \(k=m\) in the formula defining \(I_n\)  and invoking \eqref{In0}.  The identity \eqref{trID} follows from \eqref{detID} by taking the determinant of both sides of the expression
\[
1 + G(n+1,n) = U_nG(n,n).
\]

Next we turn to the proof of \eqref{AL nD} through \eqref{AL nU2}. As \(|z|\geq 2\) we may use \eqref{g11} to show that
\eq{Non-deg}{
|G_{11}(n,n)|\gtrsim \tfrac1{|z|}\qtq{and}|1 + G_{11}(n+1,n)|\gtrsim1,
}
uniformly in \(n\), whenever \(\delta>0\) is sufficiently small. Consequently, the denominators appearing in \eqref{AL nD} through \eqref{AL nU2} do not vanish.

To prove \eqref{AL nD}, we use \eqref{In} with \(m=n\) together with \eqref{In0} to obtain
\[
G_{11}(n,n)G_{22}(n,k) - G_{12}(n,k)G_{21}(n,n) = \bbo_{k<n}\tfrac{G_{22}(n+1,k)}{1 - \alpha_n\beta_n}.
\]
As \(G\) is a right inverse for \(\mb L\), we then have
\begin{align*}
&G_{11}(n,n)G_{12}(n+1,k)\\
&\qquad = zG_{11}(n,n)G_{12}(n,k) + \alpha_nG_{11}(n,n)G_{22}(n,k)\\
&\qquad = zG_{11}(n,n)G_{12}(n,k) + \alpha_n\Bigl[G_{12}(n,k)G_{21}(n,n) + \bbo_{k<n}\tfrac{G_{22}(n+1,k)}{1 - \alpha_n\beta_n}\Bigr]\\
&\qquad = \Bigl[1 + G_{11}(n+1,n)\Bigr]G_{12}(n,k) + \bbo_{k<n}\tfrac{\alpha_nG_{22}(n+1,k)}{1 - \alpha_n\beta_n},
\end{align*}
which gives \eqref{AL nD}.  The proof of \eqref{AL nU} is similar, using \eqref{Out} with $m=n$ together with \eqref{Jn0} to obtain
\[
G_{11}(n,n)G_{22}(k,n) - G_{12}(n,n)G_{21}(k,n) = \bbo_{k>n}\tfrac{G_{22}(k,n-1)}{1 - \alpha_n\beta_n},
\]
then using that \(G\) is a left inverse for \(\mb L\), and finally replacing \(n\) by \(n+1\).

The proof of \eqref{AL nD2} and \eqref{AL nU2} is essentially identical to the proof of \eqref{AL nD} and \eqref{AL nU}, except we replace \(I_n\), \(J_n\) by
\begin{align*}
\widetilde I_n(m,k) &= G_{11}(n,m)G_{21}(n,k) - G_{11}(n,k)G_{21}(n,m),\\
\widetilde J_n(m,k) &= G_{11}(m,n)G_{12}(k,n) - G_{12}(m,n)G_{11}(k,n),
\end{align*}
and compute
\begin{align*}
\widetilde I_{n+1}(m,k) &= (1 - \alpha_n\beta_n)\widetilde I_n(m,k) - \delta_{nk}G_{21}(n+1,m) - \delta_{nm}G_{21}(n+1,k),\\
\widetilde J_{n-1}(m,k) &= (1 - \alpha_n\beta_n)\widetilde J_n(m,k) - \delta_{nk}G_{12}(m,n-1) - \delta_{nm}G_{12}(k,n-1).
\end{align*}
Again setting \(m = n\), we obtain the identities
\begin{align*}
G_{11}(n,n)G_{21}(n,k) - G_{11}(n,k)G_{21}(n,n) &= \bbo_{k<n}\tfrac{G_{21}(n+1,k)}{1 - \alpha_n\beta_n},\\
G_{11}(n,n)G_{12}(k,n) - G_{12}(n,n)G_{11}(k,n) &= \bbo_{k>n}\tfrac{G_{12}(k,n-1)}{1 - \alpha_n\beta_n},
\end{align*}
and arguing as before, the first gives \eqref{AL nD2}, and the second yields \eqref{AL nU2}.
\end{proof}

Having constructed the Green's function for \(\mb L\), we now turn to the perturbation determinant:

\begin{lem}\label{l:logging}
There exists \(\delta>0\) so that the series
\begin{equation}\label{AL alpha}
- \log \det \bigl[ \mathbf L \, \mathbf L_0^{-1} \bigr]  = \sum_{\ell=1}^\infty \tfrac{1}{\ell} (-1)^{\ell} \tr\Bigr\{\bigl[\Gamma \Lambda\bigr]^\ell \Bigr\}
\end{equation}
converges uniformly for \(|z|\geq2\) and $\alpha_m\in B_\delta^z$. Further, we have the identity
\begin{align}
 z\partial_z \log \det \bigl[ \mathbf L \, \mathbf L_0^{-1} \bigr]  &= \tr \bigl\{ \bigl(\mathbf L^{-1} - \mathbf L_0^{-1}\bigr) \mathbf S\sigma_3\bigr\}.\label{AL diff z}
\end{align}
\end{lem}

\begin{proof}
Convergence of the series \eqref{AL alpha} follows from \eqref{AL HS small}. This also guarantees that we can compute the left-hand side of \eqref{AL diff z} by summing the derivatives of the summands in  \eqref{AL alpha}.  Thus, \eqref{AL diff z} follows from
\begin{align*}
z \partial_z \Lambda = \Lambda - \Lambda S(S-z^{-1})^{-1}, \qquad
z \partial_z \Gamma  = - \Gamma - \Gamma S(z-S)^{-1},
\end{align*}
and the resolvent expansions derived in Lemma~\ref{L:AL GF}
\end{proof}

Using that $G$ inverts $\mathbf L$, we have
\begin{align}\label{G inverts L}
U_{n}G(n,n) = 1+ G(n+1,n)=G(n+1,n+1)U_{n+1}.
\end{align}
Employing the top left entries in \eqref{G inverts L} as well as \eqref{trID}, we may write
\begin{align}
\rho_n &= -\tfrac12\log\Bigl[1 - 2\tfrac{\alpha_nG_{21}(n,n)}{2 + \gamma_n}\Bigr]-\tfrac12\log\Bigl[1 - 2\tfrac{\beta_{n+1}G_{12}(n+1,n+1)}{2 + \gamma_n}\Bigr]\label{rho-alt-1}\\
&= \tfrac12\log\Biggl[\tfrac{\bigl[1 + G_{11}(n+1,n)\bigr]^2}{G_{11}(n,n) G_{11}(n+1,n+1)}\Biggr] - \log z,\label{rho-alt}
\end{align}
where these expressions are seen to be well-defined using \eqref{Non-deg}. While RHS\eqref{rho-alt-1} demonstrates the link to the continuum case, the expression RHS\eqref{rho-alt} will prove to be the most useful in our proof of Theorem~\ref{T:AL}.

\begin{proof}[Proof of Theorem~\ref{T:AL}] From \eqref{trID} we obtain
\[
\gamma_n = 2G_{11}(n+1,n).
\]
Thus \eqref{g11} shows that \(\gamma_n\in \ell^1\). Similarly, recalling \eqref{g12}, \eqref{g21}, and \eqref{Non-deg} we see that \(\rho_n\in \ell^1\).

Next we turn to the derivation of the currents. Applying \eqref{AL dtG} we obtain
\[
\partial_t\gamma_n = 2i\bigl[z\alpha_{n+1} - z^{-1}\alpha_n\bigr]G_{21}(n+1,n) - 2i\bigl[z\beta_n - z^{-1}\beta_{n+1}\bigr]G_{12}(n+1,n).
\]
As $G$ inverts $\mathbf L$, we have
\[
U_{n+1}G(n+1,n) = G(n+2,n)\qtq{and} G(n+1,n)U_n = G(n+1,n-1).
\]
Using the diagonal entries in these identities we get \(\partial_t \gamma_n={}^\gamma\!j_{n+1}-{}^\gamma\!j_n\). The fact that \({}^\gamma\!j_n\in \ell^1\) follows from \eqref{g11} and \eqref{g22}.

Recalling \eqref{rho-alt}, we may use \eqref{AL dtG} to compute \(\partial_t\rho_n\) and simplify using the identities \eqref{AL nD} and \eqref{AL nU}. The fact that \(j_n\in \ell^1\) follows from \eqref{g12}, \eqref{g21}, and \eqref{Non-deg}.

The expression \eqref{AL conserve tr} follows directly from the definition of \(\gamma_n\), so it remains to prove \eqref{AL conserve det}.  Differentiating with respect to the spectral parameter, we get
\[
z\partial_zG_{11}(n,m) = - z\sum_k G_{11}(n,k)G_{11}(k,m) + z^{-1}\sum_kG_{12}(n,k)G_{21}(k,m),
\]
so from \eqref{rho-alt} we obtain
\begin{align*}
z\partial_z \rho_n &= \tfrac z2\sum_k\Bigl[\tfrac{G_{11}(n,k)G_{11}(k,n)}{G_{11}(n,n)} + \tfrac{G_{11}(n+1,k)G_{11}(k,n+1)}{G_{11}(n+1,n+1)} - 2\tfrac{G_{11}(n+1,k)G_{11}(k,n)}{1+G_{11}(n+1,n)}\Bigr]\\
&\quad - \tfrac1{2z}\sum_k\Bigl[\tfrac{G_{12}(n,k)G_{21}(k,n)}{G_{11}(n,n)} + \tfrac{G_{12}(n+1,k)G_{21}(k,n+1)}{G_{11}(n+1,n+1)} - 2\tfrac{G_{12}(n+1,k)G_{21}(k,n)}{1 + G_{11}(n+1,n)}\Bigr] - 1.
\end{align*}
We then apply \eqref{AL nD}, \eqref{AL nD2} whenever \(k\geq n\) and \eqref{AL nU}, \eqref{AL nU2} whenever \(k<n\) to replace the last term in each summand. This yields the expression
\begin{align*}
z\partial_z \rho_n &= \psi_{n+1} - \psi_n + z\tfrac{G_{11}(n+1,n)G_{11}(n,n+1)}{G_{11}(n+1,n+1)} + z\tfrac{G_{11}(n,n)}{1 + G_{11}(n+1,n)}\\
&\quad - z^{-1}\tfrac{G_{12}(n+1,n)G_{21}(n,n+1)}{G_{11}(n+1,n+1)}  - 1,
\end{align*}
where we define
\[
\psi_n := \tfrac 12\sum_k\sgn(k+\tfrac12 - n)\Bigl[z\tfrac{G_{11}(n,k)G_{11}(k,n)}{G_{11}(n,n)} - \tfrac1z\tfrac{G_{12}(n,k)G_{21}(k,n)}{G_{11}(n,n)}\Bigr] - \tfrac12.
\]

To further simplify our expression for \(z\partial_z \rho_n\), we use \eqref{AL nD}, \eqref{AL nU}, \eqref{AL nU2}, and \eqref{detID} to write
\begin{align*}
&z\tfrac{G_{11}(n+1,n)G_{11}(n,n+1)}{G_{11}(n+1,n+1)} + z\tfrac{G_{11}(n,n)}{1 + G_{11}(n+1,n)} - z^{-1}\tfrac{G_{12}(n+1,n)G_{21}(n,n+1)}{G_{11}(n+1,n+1)}\\
&\qquad = zG_{11}(n,n) -z^{-1}G_{22}(n,n).
\end{align*}
The identity \eqref{G inverts L} gives the expressions
\begin{align*}
zG_{11}(n,n) - z^{-1}G_{22}(n,n) &= \beta_n G_{12}(n,n) - \alpha_n G_{21}(n,n) + \bigl[1 + \gamma_n\bigr]\\
&= \alpha_n G_{21}(n,n) - \beta_n G_{12}(n,n) + \bigl[1 + \gamma_{n-1}\bigr],
\end{align*}
so, after taking the mean of the right-hand sides, we arrive at
\[
z\partial_z \rho_n = \psi_{n+1} - \psi_n + \tfrac12\bigl[\gamma_n + \gamma_{n-1}\bigr].
\]

Recalling \eqref{Non-deg}, we may apply \eqref{g12}--\eqref{g11} to see that \(\psi_n\in \ell^1\). This gives
\[
z\partial_z\sum_n\rho_n = \sum_n\gamma_n,
\]
which combined with \eqref{AL conserve tr} and \eqref{AL diff z} yields
\begin{align}\label{deriv rho}
\partial_z\sum_n\rho_n = \partial_z\log\det\bigl[\mb L\mb L_0^{-1}\bigr].
\end{align}
Using Lemma~\ref{l:logging} and \eqref{AL HS}, we find that
\[
\lim\limits_{z\to\infty}\log\det\bigl[\mb L\mb L_0^{-1}\bigr]= 0 = \lim\limits_{z\to\infty}\sum_n\rho_n.
\]
Thus, integrating \eqref{deriv rho} with respect to the spectral parameter we obtain \eqref{AL conserve det}.
\end{proof}

As in the continuum case, the quadratic part of the current \(j_n\) associated to \(\rho_n\) has good coercivity properties.  This is what is needed for proving local smoothing.

Let us define the Fourier transform of a sequence, say $\alpha_n$, via
$$
\hat\alpha(\theta):= \sum_n \alpha_ne^{in\theta}.
$$
Then, writing \(j^{[2]}\) for the quadratic part of the current, we have positive-definite expression
\begin{align}\label{E:Im j}
\sum_n \Im j_n^{[2]} = \mp \int_0^{2\pi}\frac{2z^2\sin^2(\theta)}{|z^2-e^{i\theta}|^2} |\hat\alpha(\theta)|^2\,\frac{d\theta}{2\pi}.
\end{align}
This is the analogue of \cite[Lemma~4.8]{HGKV} relevant to \eqref{AL eqn}.

The real part of $\rho_n$ is not itself coercive in the same sense. However, this is readily remedied by switching to
$$
\tilde \rho_n = \rho_n - \tfrac12\log\bigl(1-\alpha_n\beta_n) \qtq{with current} \tilde \jmath_n = j_n +\tfrac{i}2(\alpha_n\beta_{n-1}-\alpha_{n-1}\beta_{n}).
$$
Specifically, we find that
\begin{align}\label{E:Re rho}
\sum_n \Re \tilde\rho_n^{[2]} = \pm \int_0^{2\pi}\frac{z^4 - 1}{2|z^2-e^{i\theta}|^2} |\hat\alpha(\theta)|^2\,\frac{d\theta}{2\pi}.
\end{align}
Advantageously, this modification does not affect the coercivity \eqref{E:Im j} at all; indeed, $\sum \Im (j_n-\tilde\jmath_n) =0$.  Moreover, this modification drops out in the continuum limit.

\bibliographystyle{habbrv}
\bibliography{refs}

\end{document}